\documentclass{article}
\usepackage{amssymb,amsmath,graphicx,ucs}
\usepackage[utf8x]{inputenc}
\usepackage[english]{babel}

\newtheorem{conjecture}{Conjecture}
\newtheorem{theorem}{Theorem}

\newenvironment{proof}{\noindent \emph{Proof. }}{\hfill \hbox{\rlap{$\sqcap$}$\sqcup$}\\}

\title{Icosahedral Tilings Revisited\thanks{This work was supported by the ANR project QuasiCool (ANR-12-JS02-011-01)}}
\author{
Nicolas Bédaride
\footnote{Aix Marseille Univ., CNRS, Centrale Marseille, I2M, UMR 7373, 13453 Marseille, France.}
\and
Thomas Fernique
\footnote{Univ. Paris 13, CNRS, Sorbonne Paris Cité, UMR 7030, 93430 Villetaneuse, France.}
}
\date{}

\begin{document}

\maketitle

\begin{abstract}
Icosahedral tilings, although non-periodic, are known to be characterized by their configurations of some finite size \cite{katz}.
This characterization has also been expressed in terms of a simple {\em alternation condition} \cite{socolar}.
We provide an alternative proof -- shorter and arguably easier -- of this fact.
We moreover conjecture that the alternation condition can be weakened.
\end{abstract}

\paragraph{Golden rhombohedra.}
Let us denote by $\pm\vec{v}_1,\ldots,\pm\vec{v}_6$ the vertices of the icosahedron (Fig.~\ref{fig:icosahedron}).
These vectors generate $\binom{6}{3}$ rhomboedra, called {\em golden rhombohedra}:
$$
T_{ijk}:=\{\lambda\vec{v}_i+\mu\vec{v}_j+\nu\vec{v}_k~|~0\leq\lambda,\mu,\nu\leq 1\}.
$$
Up to isometry, there are only two rhombohedra: the prolate (or thin)
$$
T_{123},~T_{124},~T_{136},~T_{145},~T_{156},~T_{235},~T_{246},~T_{256},~T_{345},~T_{346},
$$
and the oblate (or fat) whose volume is $\varphi$ times bigger
$$
T_{125},~T_{126},~T_{134},~T_{135},~T_{146},~T_{234},~T_{236},~T_{245},~T_{356},~T_{456}.
$$
\begin{figure}[hbtp]
\centering
\includegraphics[width=0.8\textwidth]{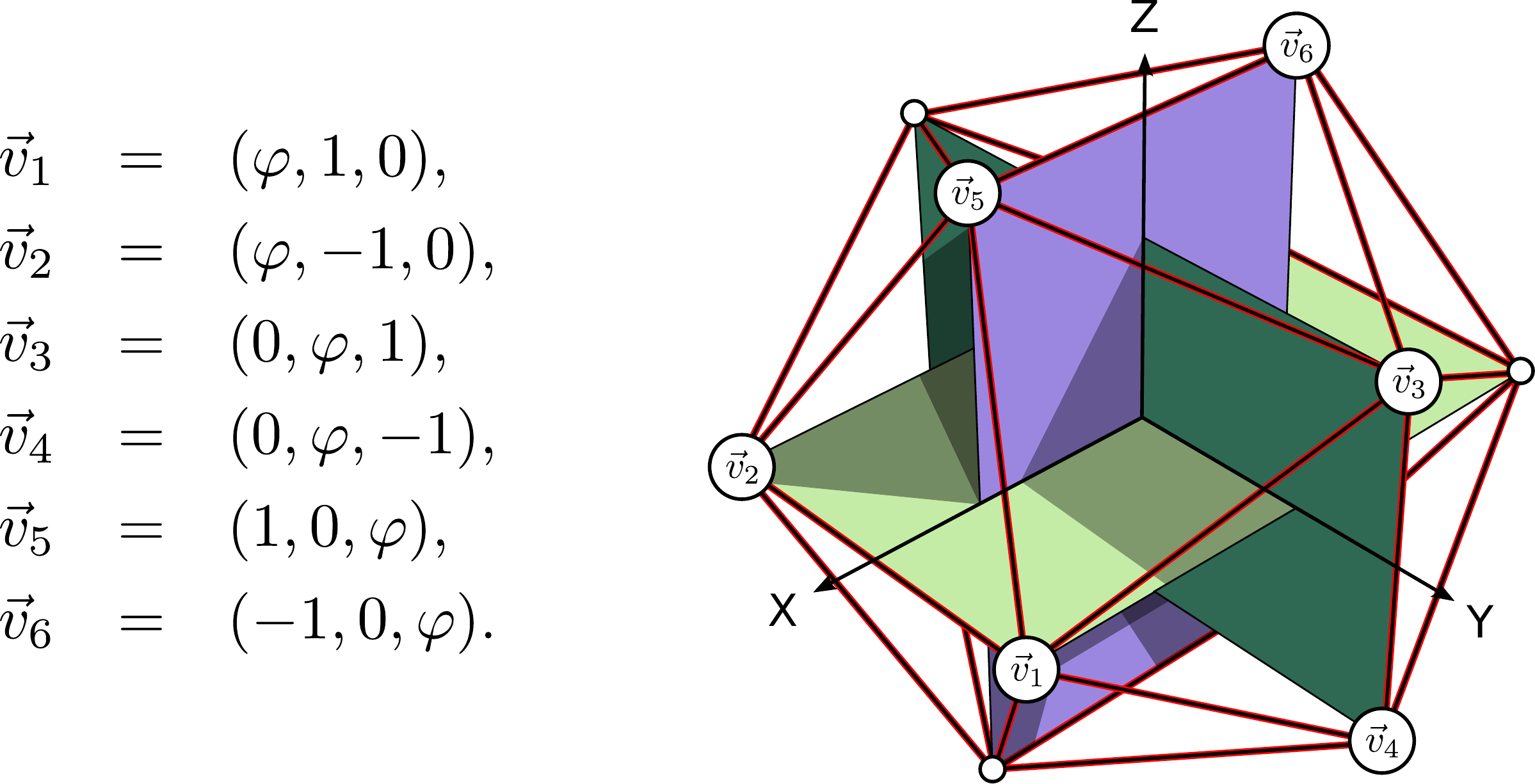}
\caption{The icosahedron, with $\varphi=\frac{1+\sqrt{5}}{2}$.}
\label{fig:icosahedron}
\end{figure}

\paragraph{Golden tilings and planarity.}
A {\em golden tiling} is a covering of $\mathbb{R}^3$ by translated copies of the golden rhombohedra which can pairwise intersect only on a whole face, a whole edge or a vertex ({\em face-to-face} condition).
Let $\vec{e}_1,\ldots,\vec{e}_6$ be the canonical basis of $\mathbb{R}^6$.
A golden tiling is {\em lifted} in $\mathbb{R}^6$ as follows: an arbitrary vertex is first mapped onto the origin of $\mathbb{R}^6$, then each rhombohedra $T_{ijk}$ is mapped onto the $3$-dim. face of a unit hypercube of $\mathbb{Z}^6$ generated by $\vec{e}_i$, $\vec{e}_j$ and $\vec{e}_k$, with two rhombohedra adjacent along a face generated by $\vec{v}_i$ and $\vec{v}_j$ being mapped onto two $3$-dim. faces adjacent along a $2$-dim. face generated by $\vec{e}_i$ and $\vec{e}_j$.
This defines a ``stepped'' manifold of codimension $3$ in $\mathbb{R}^6$ (unique up to the choice of the initial vertex).
By extension, golden tilings are said to have codimension $3$ -- they are also called $6\to 3$ tilings.
A golden tiling is moreover said to be {\em planar} if there is $t\geq 1$ and a $3$-dim. vectorial subspace $E$ of $\mathbb{R}^6$ such that the tiling can be lifted in the tube $E+[0,t]^6$.
The subspace $E$ is called the {\em slope} of the tiling and $t$ its {\em thickness}.\\

\paragraph{Grassmann coordinates.}
Recall (see, {\em e.g.}, \cite{HP}) that the Grassmann coordinates $G_{ijk}$ of a $3$-dim. space are the $3\times 3$ minors of the matrix formed by three vectors $\vec{w}_1$, $\vec{w}_2$ and $\vec{w}_3$ generating this space:
$$
G_{ijk}=\det \left(\begin{array}{ccc}
w_{1i} & w_{2i} & w_{3i}\\
w_{1j} & w_{2j} & w_{3j}\\
w_{1k} & w_{2k} & w_{3k}
\end{array}\right).
$$
Permuting two indices in a Grassmann coordinate changes it sign and a Grassmann coordinate with two identical indices is equal to zero.
The product of two Grassmann coordinates is equal to the sum of the products of these two Grassmann coordinates where a fixed index of the first coordinate has been exchanged with any possible index of the second one:
\begin{eqnarray*}
G_{ijk}G_{abc}
&=&G_{ajk}G_{ibc}+G_{bjk}G_{aic}+G_{cjk}G_{abi}\\
&=&G_{iak}G_{jbc}+G_{ibk}G_{ajc}+G_{ick}G_{abj}\\
&=&G_{ija}G_{kbc}+G_{ijb}G_{akc}+G_{ijc}G_{abk}.
\end{eqnarray*}
These quadratic relations are called {\em Plücker relations}.
They characterize the real numbers that are Grassmann coordinates of some $3$-dim. subspace of $\mathbb{R}^6$.\\

\paragraph{Subperiods.}
Let $\pi_{ijkl}$ denotes the orthogonal projection onto the space generated by $\vec{e}_i$, $\vec{e}_j$, $\vec{e}_k$ and $\vec{e}_l$.
A {\em $ijkl$-shadow} of a golden tiling is the image of this tilings under $\pi_{ijkl}$.
A vector of $\mathbb{Z}^4$ is a {\em $ijkl$-subperiod} of a golden tiling if it is a period of its $ijkl$-shadow.
By extension, one call $ijkl$-subperiod any vector of $\mathbb{R}^6$ whose image under $\pi_{ijkl}$ is a period of its $ijkl$-shadow.
One checks that a planar golden tiling admits $(x_i,x_j,x_k,x_l)$ as a $ijkl$-subperiod iff the Grassmann coordinates of its slope satisfy
$$
x_iG_{jkl}-x_jG_{ikl}+x_kG_{ijl}-x_lG_{ijk}=0.
$$

\paragraph{Icosahedral tilings.}
Icosahedral tilings are the planar golden tilings whose slope is generated by
\begin{eqnarray*}
\vec{w}_1&=&(\varphi,\varphi,0,0,1,-1),\\
\vec{w}_2&=&(1,-1,\varphi,\varphi,0,0),\\
\vec{w}_3&=&(0,0,1,-1,\varphi,\varphi).
\end{eqnarray*}
In other words, $\vec{w}_i$ is the vector of $\mathbb{R}^6$ whose $j$-th entry is the $i$-th one of $\vec{v}_j$.
One thus has $G_{ijk}=\det(\vec{v}_i,\vec{v}_j,\vec{v}_k)$, that is, $G_{ijk}$ is, up to the sign, equal to the volume of the golden rhombohedron $T_{ijk}$.
More precisely:
\begin{equation}\label{eq:prolate_alternation}
G_{213}=G_{124}=G_{136}=G_{145}=G_{516}=G_{235}=G_{246}=G_{256}=G_{435}=G_{346},
\end{equation}
\begin{equation}\label{eq:oblate_alternation}
G_{215}=G_{216}=G_{314}=G_{135}=G_{146}=G_{324}=G_{236}=G_{245}=G_{536}=G_{546}.
\end{equation}
One checks that, together with the Plücker relations\footnote{Actually, it suffices to consider the Plücker relation $G_{123}G_{345}=G_{423}G_{315}+G_{523}G_{341}$.}, the above equations yield a system of polynomial equations with two solutions: the slope of icosahedral tilings and its algebraic conjugate.
One checks that $\vec{w}_1$ is a $1234$- and  $3456$-subperiod, $\vec{w}_2$ is a $1256$- and $3456$-subperiod and $\vec{w}_3$ is a $1234$- and $1256$-subperiod.\\

\paragraph{Alternation condition and planarity.}
A {\em $ij$-worm} of a golden tiling is a biinfinite sequence of rhombohedra of this tiling, each one being adjacent to the following one along a face generated by $\vec{v}_i$ and $\vec{v}_j$.
In any worm, the rhombohedra of a given type (oblate or prolate) can occurs in two orientations; one says that a golden tiling has {\em alternation} if these orientations perfectly alternate (for each type of rhombohedra).

\begin{theorem}
A golden tiling with alternation is planar.
\end{theorem}

\begin{proof}
Let $E$ be the slope of the icosahedral tilings and $E'$ be its algebraic conjugate.
One has $E\cap E'=\{\vec{0}\}$.
Let $\vec{r}_i$ denotes the algebraic conjugate of $\vec{w}_i$.
Let $\mathcal{S}$ be a golden tiling satisfying the alternation condition, hence having (at least) the subperiods of icosahedral tilings.
One parametrizes $\mathcal{S}$ as follows:
$$
\mathcal{S}=\{\vec{x}+z_1(\vec{x})\vec{r}_1+z_2(\vec{x})\vec{r}_2+z_3(\vec{x})\vec{r}_3~|~\vec{x}\in E\}
$$
First, consider the projection $\pi_{1256}$:
$$
\pi_{1256}(\mathcal{S})=\{\pi_{1256}(\vec{x})+z_1(\vec{x})\pi_{1256}(\vec{r}_1)+z_2(\vec{x})\pi_{1256}(\vec{r}_2)+z_3(\vec{x})\pi_{1256}(\vec{r}_3)~|~\vec{x}\in E\}.
$$
Recall that $\vec{w}_2$ and $\vec{w}_3$, hence $\vec{r}_2$ and $\vec{r}_3$, are $1256$-subperiods of icosahedral tilings, hence of $\mathcal{S}$.
The intersection of $\pi_{1256}(\mathcal{S})$ and $\pi_{1256}(E')$ is thus a $2$-dim. surface which is both $\pi_{1256}(\vec{r}_2)$ and $\pi_{1256}(\vec{r}_3)$ periodic, thus stays at bounded distance from the plane generated by $\pi_{1256}(\vec{r}_2)$ and $\pi_{1256}(\vec{r}_3)$.
Now, when $\vec{x}$ moves along $\vec{w}_2$ or $\vec{w}_3$, the quantity
$$
\pi_{1256}(\vec{x})+z_2(\vec{x})\pi_{1256}(\vec{r}_2)+z_3(\vec{x})\pi_{1256}(\vec{r}_3)
$$
remains a plane generated by $\pi_{1256}(\vec{r}_2)$ and $\pi_{1256}(\vec{r}_3)$.
Since this plane does not contain $\pi_{1256}(\vec{r}_1)$, this ensures that $z_1$ has bounded fluctuations in the directions $\vec{w}_2$ and $\vec{w}_3$.
There is thus a real function $f$ such that $z_1(a\vec{w}_1+b\vec{w}_2+c\vec{w}_3)\equiv f(a)$, where $f\equiv g$ means that $f-g$ is uniformly bounded.\\
One then similarly gets, by considering the projection $\pi_{1234}$, that $z_2$ has bounded fluctuations in the directions $\vec{w}_1$ and $\vec{w}_3$.
There is thus a real function $g$ such that $z_2(a\vec{w}_1+b\vec{w}_2+c\vec{w}_3)\equiv g(b)$.\\
Last, one similarly gets, by considering the projection $\pi_{3456}$, that $z_3$ has bounded fluctuations in the directions $\vec{w}_1$ and $\vec{w}_2$.
There is thus a real function $h$ such that $z_3(a\vec{w}_1+b\vec{w}_2+c\vec{w}_3)\equiv h(c)$.\\
Now, let us introduce the vectors defined by
$$
2\vec{w}_4:=-\vec{w}_1+\varphi\vec{w}_2+(1-\varphi)\vec{w}_3,
$$
$$
2\vec{w}_5:=\vec{w}_1+(2-\varphi)\vec{w}_2+(1-\varphi)\vec{w}_3.
$$
One checks that $\vec{w}_4$ and $\vec{w}_5$ are independent $1356$-subperiods of icosahedral tilings.
Let $\vec{r}_4$ and $\vec{r}_5$ denotes their algebraic conjugates.
Consider the projection
$$
\pi_{1356}(\mathcal{S})=\{\pi_{1356}(\vec{x})+z_1(\vec{x})\pi_{1356}(\vec{r}_1)+z_2(\vec{x})\pi_{1356}(\vec{r}_2)+z_3(\vec{x})\pi_{1356}(\vec{r}_3)~|~\vec{x}\in E\}.
$$
One computes
$$
\vec{w}_1=(2-\varphi)\vec{w}_3+(\varphi-2)\vec{w}_4+\varphi\vec{w}_5,
$$
$$
\vec{w}_2=(\varphi-1)\vec{w}_3+\vec{w}_4+\vec{w}_5,
$$
and the same relations holds for the $\vec{r}_i$'s.
One can thus write
$$
\pi_{1356}(\mathcal{S})=\{\pi_{1356}(\vec{x})+\vec{u}(\vec{x})+[(2-\varphi)z_1+(\varphi-1)z_2+z_3](\vec{x})\pi_{1356}(\vec{r}_3)~|~\vec{x}\in E\}.
$$
where $\vec{u}(\vec{x})$ is in the plane generated by $\vec{r}_4$ and $\vec{r}_5$.
When $\vec{x}$ moves along $\vec{w}_4$ and $\vec{w}_5$, the above projection remains at bounded distance from a plane generated by $\pi_{1456}(\vec{r}_4)$ and $\pi_{1456}(\vec{r}_5)$ (because $\vec{w}_4$ and $\vec{w}_5$ -- hence $\vec{r}_4$ and $\vec{r}_5$ -- are $1356$-subperiods of icosahedral tilings).
This ensures that $(2-\varphi)z_1+(\varphi-1)z_2+z_3$ has bounded fluctuations in the directions $\vec{w}_4$ and $\vec{w}_5$.
Thus, for any $(\lambda,\mu,\nu)$ such that $\lambda\vec{w}_1+\mu\vec{w}_2+\nu\vec{w}_3$, $\vec{w}_4$ and $\vec{w}_5$ are linearly independent, there is a real function $\psi$ such that $[(2-\varphi)z_1+(\varphi-1)z_2+z_3](a\vec{w}_1+b\vec{w}_2+c\vec{w}_3)\equiv \psi(\lambda a+\mu b+\nu c)$, 
For $(\lambda,\mu,\nu)=(2-\varphi,\varphi-1,1)$, this yields
$$
(2-\varphi)f(a)+(\varphi-1)g(b)+h(c)\equiv\psi((2-\varphi)a+(\varphi-1)b+c).
$$
From this it is not hard to deduce the linearity (up to bounded fluctuations) of $\psi$, then $f$, $g$ and $h$, thus $z_1$, $z_2$ and $z_3$.
The planarity of $\mathcal{S}$ follows.
\end{proof}

\paragraph{Weak alternation.}
Actually, one checks that only one of the equations (\ref{eq:prolate_alternation}) or (\ref{eq:oblate_alternation}) suffices, together with the Plücker relations, to characterize the Grassmann coordinates of the slope of the icosahedral tilings.
That is, the sole alternation of one type of rhomboedron suffices to characterize the icosahedral tilings {\em among planar tilings}.
Does such a {\em weak alternation} also enforces planarity?
It is actually not hard to provide counter-examples using only the rhomboedra which are not forced to alternate.
We therefore consider {\em nondegenerate} alternation: there is no infinite run of rhomboedra all of the same type.

\begin{conjecture}
  A golden tiling with nondegenerate weak alternation is planar.
\end{conjecture}

  
\end{document}